\documentclass[12pt,reqno]{amsart}

\usepackage{afterpage}
\usepackage{amsfonts}
\usepackage{amsmath}
\usepackage{amssymb}
\usepackage{amsthm}
\usepackage{bbding}
\usepackage{bbm}
\usepackage{bm}
\usepackage{caption}
\usepackage{contour}
\usepackage{dsfont}
\usepackage{enumerate}
\usepackage{enumitem}
\usepackage{esvect}
\usepackage{float}
\usepackage[top=3cm, bottom=2.5cm, left=2.5cm, right=2.5cm]{geometry}
\usepackage{graphicx}
\usepackage{lscape}
\usepackage{mathrsfs}
\usepackage{mathtools}
\usepackage{physics}
\usepackage{relsize}
\usepackage{slantsc}
\usepackage{thmtools}
\usepackage{tikz}
\usepackage[normalem]{ulem}
\usepackage{url}
\usepackage{xcolor}
\usepackage{hyperref}

\makeatletter
\def\th@plain{%
    \thm@notefont{}
    \itshape
}
\def\th@definition{%
    \thm@notefont{}
    \normalfont
}
\makeatother

\newtheorem{theorem}{Theorem}[section]
\newtheorem{lemma}[theorem]{Lemma}
\newtheorem{proposition}[theorem]{Proposition}
\newtheorem{corollary}[theorem]{Corollary}

\theoremstyle{definition}

\theoremstyle{remark}

\theoremstyle{theorem}

\theoremstyle{remark}
\newtheorem*{uremark}{Remark}

\DeclareMathOperator*{\vol}{vol}

\numberwithin{equation}{section}

\begin{document}

\title[Equidistribution and Diophantine approximation]{Equidistribution on homogeneous spaces and the distribution of approximates in Diophantine approximation}

\thanks{A.\ G.\ was supported by a grant from the Indo-French Centre for the Promotion of Advanced Research; a Department of Science and Technology, Government of India Swarnajayanti fellowship and a MATRICS grant from the Science and Engineering Research Board.}

\author{Mahbub Alam}
\address{\textbf{Mahbub Alam} \\
    School of Mathematics,
    Tata Institute of Fundamental Research, Mumbai, India 400005}
\email{mahbub@math.tifr.res.in}

\author{Anish Ghosh}
\address{\textbf{Anish Ghosh} \\
    School of Mathematics,
    Tata Institute of Fundamental Research, Mumbai, India 400005}
\email{ghosh@math.tifr.res.in}

\date{}

\begin{abstract}
    The present paper is concerned with equidistribution results for certain flows on homogeneous spaces and related questions in Diophantine approximation.
    Firstly, we answer in the affirmative, a question raised by Kleinbock, Shi and Weiss~\cite{kleinbockshiweiss17} regarding equidistribution of orbits of arbitrary lattices under diagonal flows and with respect to unbounded functions.
    We then consider the problem of Diophantine approximation with respect to rationals in a fixed number field.
    We prove a number field analogue of a famous result of W.\ M.\ Schmidt which counts the number of approximates to Diophantine inequalities for a certain class of approximating functions.
    Further we prove ``spiraling'' results for the distribution of approximates of Diophantine inequalities in number fields.
    This generalizes the work of Athreya, Ghosh and Tseng~\cite{athreyaghoshtseng15, athreyaghoshtseng14} as well as Kleinbock, Shi and Weiss~\cite{kleinbockshiweiss17}.
\end{abstract}

\maketitle

\section{Introduction}\label{}

In this paper, we establish several results relating to the distribution of approximates in Diophantine approximation.
To motivate our results, let us recall the starting point of Diophantine approximation, namely the following corollary to Dirichlet's theorem.

\begin{theorem}\label{}
    For every $\bm{x} \in \mathbb{R}^{m}$ $(m \geq 1)$, there exist infinitely many $(\bm{p}, q) \in \mathbb{Z}^{m} \times \mathbb{Z}_+$ such that
    \begin{equation}\label{dirichlet}
        \|q\bm{x} - \bm{p}\| < c_m|q|^{-1/m}.
    \end{equation}
\end{theorem}

If $\norm{ \cdot }$ is taken to be the supremum norm then $c_m$ can be taken to be 1.
In~\cite{athreyaghoshtseng15}, Athreya, Ghosh and Tseng considered the problem of `spiraling' of approximates connected to the Diophantine inequality above.
Let
\[
    \theta(\bm{p}, q) := \frac{q\bm{x} - \bm{p}}{\|q\bm{x} - \bm{p}\|} \in \mathbb{S}^{m-1}.
\]
Given $A \subseteq \mathbb{S}^{m-1}$, $T >0$, they considered the counting functions
\[
    N(\bm{x}, T) = \#\{(\bm{p}, q) \in \mathbb{Z}^{m} \times \mathbb{Z}_+ : \|q\bm{x} - \bm{p}\| <  c_m|q|^{-1/m}, 0 < q \le T\}
\]
and
\[
    N(\bm{x}, T, A) = \#\{(\bm{p}, q) \in \mathbb{Z}^{m} \times \mathbb{Z}_+ : \|q\bm{x} - \bm{p}\| < c_m|q|^{-1/m}, 0 < q \le T, \theta(\bm{p}, q) \in A \},
\]
and proved:

\begin{theorem}[\cite{athreyaghoshtseng15} Theorem $1.1$]
    For $A \subseteq \mathbb S^{m-1}$ a measurable subset, and for almost every $\bm{x} \in \mathbb{R}^m$,
    \[
        \lim_{T \rightarrow \infty} \frac{N(\bm{x},T, A)}{N(\bm{x}, T)} = \vol(A).
    \]
    Here $\vol := \vol_{\mathbb{S}^{m-1}}$ is the Lebesgue probability measure on $\mathbb{S}^{m-1}$.
\end{theorem}

The result above is closely connected to equidistribution results on the space of unimodular lattices in $\mathbb{R}^{d}$ $(d = m+1)$ and its proof involves a study of spherical averages of Siegel transforms of functions on the space of lattices.
Subsequently, weighted and multiplicative versions of the above result were established in~\cite{athreyaghoshtseng14} (Theorems $1.5$ and $1.6$).
The study of spiraling was then taken up in~\cite{kleinbockshiweiss17}, where several results including a stronger weighted version of the above Theorem was established as a consequence of equidistribution results for diagonal orbits of points on unipotent orbits on the space of lattices.\\

In this paper, among other results, we generalize the above weighted spiraling of approximates to number fields.
We follow the approach of Kleinbock, Shi and Weiss and use an equidistribution result of Shi~\cite{shi17}.
We also answer a question raised by Kleinbock, Shi and Weiss regarding the equidistribution of orbits of certain flows on homogeneous spaces with respect to unbounded functions.
While these results are proved for an arbitrary number field, they are new even for $\mathbb{Q}$ which is the case they were asked for in~\cite{kleinbockshiweiss17}.
Another ingredient in the proof is a number field analogue of famous result of W.\ M.\ Schmidt~\cite{schmidt60} providing an asymptotic count for the number of solutions to Diophantine inequalities.
We believe this result to be of independent interest.
We now describe our results in detail.

\section{Main results}\label{}

Let $K$ be a number field of degree $k$ and let $S$ be the set of all normalized non-conjugate Archimedean valuations on $K$.
Let $S_\mathbb{R} \subseteq S$ be the set of real valuations and $S_\mathbb{C} \subseteq S$ be the set of non-conjugate complex valuations, then $\#(S_\mathbb{R}) + \#(S_\mathbb{C}) = \#S$ and $\#(S_\mathbb{R}) + 2 \cdot \#(S_\mathbb{C}) = k$.
We denote the ring of integers of $K$ by $\mathcal{O}$.
For each $v \in S$ denote by $\iota_v$ the corresponding embedding of $K$ into a completion $K_v$ with respect to $v$.
The Minkowski space associated with $K$ is defined by
\begin{equation}\label{eqdefnks}
    K_S := \prod_{v \in S} K_v \cong \prod_{v \ \text{real}} \mathbb{R} \times \prod_{\substack{v \\ \ \text{complex}}} \mathbb{C} \cong \mathbb{R}^{k}.
\end{equation}
The `diagonal' embedding of $K$ into $K_S$ is denoted by
\[
    \iota_S : K \to K_S, \ r \mapsto {(\iota_v(r))}_{v \in S}.
\]
Note that $\mathcal{O}$ is a lattice in $K_S$ via this embedding.
We permit ourselves a mild abuse of notation and denote the image of $\mathcal{O}$ inside $K_S$ under this embedding by $\mathcal{O}$ itself.
Take a Haar measure on $\mathbb{R}$ which is a multiple (say $\varpi$) of the Lebesgue measure, and denote by $\lambda$ the product measure on $K_S$ induced via the isomorphism $K_S \cong \mathbb{R}^{k}$.
The constant $\varpi$ is so chosen that $\mathcal{O}$ becomes a covolume 1 lattice in $K_S \cong \mathbb{R}^{k}$.

\noindent We define
\[
    \bm{x} \cdot \bm{y} := {(x_v y_v)}_{v \in S}
\]
and fix norms on $K_S$ and $K_S^{d}$ $(d \geq 2)$ as follows:
\begin{equation*}
    \norm{\bm{x}} := \max_{v \in S} |x_v|, \qquad \norm{\bm{x}}_2 := {\left(\sum\nolimits_{v \in S} |x_v|^2\right)}^{\frac{1}{2}}, \\
\end{equation*}
and
\begin{equation*}
    \norm{\vv{\bm{x}}} := \max_{1 \leq i \leq d} \norm{\bm{x}_i}, \qquad \norm{\vv{\bm{x}}}_2 := {\left(\sum\nolimits_{i=1}^{d} \norm{\bm{x}_i}_2^2\right)}^{\frac{1}{2}}
\end{equation*}
for every $\bm{x} = {(x_v)}_{v \in S}$, $\bm{y} = {(y_v)}_{v \in S}$ in $K_S$ and $\vv{\bm{x}} = (\bm{x}_1, \ldots, \bm{x}_d) \in K_S^{d}$.
Here and throughout the rest of the paper, vectors should be thought of as column vectors even though we will write them as row vectors.

Let $m, n \in \mathbb{Z}_+$ be such that $m + n = d$.
The diagonal embedding $\iota_S : K \to K_S$ can be naturally extended to matrices.
For
\[
    A = (A_v)_{v \in S}, B = (B_v)_{v \in S} \in \mathrm{M}_{m \times n}(K_S) := \prod_{v \in S} \mathrm{M}_{m \times n}(K_v),
\]
we define $AB := (A_v B_v)_{v \in S} \in \mathrm{M}_{m \times n} (K_S)$, and for
\[
    \vv{\bm{y}} = (\vv{y_v})_{v \in S} = (y_{1, v}, \ldots, y_{n, v})_{v \in S} \in \prod_{v \in S} K_v^{n} = K_S^{n},
\]
define $A\vv{\bm{y}} := (A_v\vv{y_v})_{v \in S} \in \prod_{v \in S} K_v^{m} = K_S^{m}$.
Note that $K$ acts naturally on $K_S^{d}$ as
\[
    a \vv{\bm{x}} := (\iota_S(a) \cdot \bm{x}_1, \ldots, \iota_S(a) \cdot \bm{x}_d) \ \text{for} \ a \in K \ \text{and} \ \vv{\bm{x}} = (\bm{x}_1, \ldots, \bm{x}_d) \in K_S^{d}.
\]

Let $G = G_{K} := \mathrm{SL}_{d}(K_S) \cong \prod\limits_{v \in S} \mathrm{SL}_{d}(K_v)$, then $\Gamma = \Gamma_{K} := \iota_S(\mathrm{SL}_{d}(\mathcal{O})) = \mathrm{SL}_{d}(\iota_S(\mathcal{O}))$ is a lattice in $G$.
Denote by $X = X_{K}$ the homogeneous space $G/\Gamma$.
Note that $G_\mathbb{Q} = \mathrm{SL}_{d}(\mathbb{R}), \Gamma_\mathbb{Q} = \mathrm{SL}_{d}(\mathbb{Z})$ and $X_\mathbb{Q} = \mathrm{SL}_{d}(\mathbb{R})/\mathrm{SL}_{d}(\mathbb{Z})$.
Let $\mu = \mu_{K}$ be the left Haar measure on $G$ so that the induced $G$-invariant measure on $X$, which we also denote by $\mu$, satisfies $\mu(X) = 1$.

The isomorphism $K_S \cong \mathbb{R}^{k}$ induces an embedding $\mathrm{SL}_{d}(K_S) \hookrightarrow \mathrm{SL}_{dk}(\mathbb{R})$, so that $X$ can be identified with a proper subset of $\mathrm{SL}_{dk}(\mathbb{R})/\mathrm{SL}_{dk}(\mathbb{Z})$, the moduli space of unimodular lattices in $\mathbb{R}^{dk}$.
The map $g\Gamma \mapsto g\iota_S(\mathcal{O}^{d})$ identifies $X$ with the space of discrete rank $d$ free $\mathcal{O}$-submodules of $K_S^{d}$ having basis $\{ \vv{\bm{x}}_1, \ldots, \vv{\bm{x}}_d\}$, such that for every $v \in S, \{ \vv{x}_{1, v}, \ldots, \vv{x}_{d, v}\}$ forms a parallelepiped of area 1 in $K_v^{d}$.
Such an $\mathcal{O}$-submodule of $K_S^{d}$ will be called a \emph{unimodular lattice in $K_S^{d}$}.
See~\cite{einsiedlerghoshlytle16} and~\cite{kleinbockly16} for more details.
\newline

We will consider the problem of weighted Diophantine approximation.
Accordingly, we choose `weight vectors'
\begin{gather*}
    \bm{a} = (a_{i, v})_{1 \leq i \leq m, v \in S} \in \mathbb{R}^{(m \times \#S)}_{>0} \ \text{and} \ \bm{b} = (b_{j, v})_{1 \leq j \leq n, v \in S} \in \mathbb{R}^{(n \times \#S)}_{>0} \\
    \text{such that} \ \sum_{v \in S_\mathbb{R}} \sum_{i = 1}^{m} a_{i, v} + 2 \sum_{v \in S_\mathbb{C}} \sum_{i = 1}^{m} a_{i, v} = \sum_{v \in S_\mathbb{R}} \sum_{j = 1}^{n} b_{j, v} + 2 \sum_{v \in S_\mathbb{C}} \sum_{j = 1}^{n} b_{j, v} = 1,
\end{gather*}
and consider the diagonal subgroup
\begin{equation}\label{eqgtforks}
    D = D_K := {\{g_t\}}_{t \in \mathbb{R}}, \ \text{where} \ g_t = (\operatorname{diag}(e^{a_{1, v} t}, \ldots, e^{a_{m, v} t}, e^{-b_{1, v} t}, \ldots, e^{-b_{n, v} t}))_{v \in S}.
\end{equation}
Denote $\{g_t\}_{t>0}$ by $D^+$.
Given an $L^1$-function $\varphi$ on $X$, following~\cite{kleinbockshiweiss17}, we will say that $\Lambda \in X$ is $(D^+, \varphi)$-\emph{generic} (or $\Lambda \in X$ is \emph{Birkhoff generic for the action of $\{g_t\}$ with respect to $\varphi$}) if
\begin{equation}\label{eqdvphgen}
    \lim_{T \to \infty} \frac{1}{T} \int_{0}^{T} \varphi(g_t \Lambda) \, \mathrm{d}t = \int_{X} \varphi \, \mathrm{d}\mu.
\end{equation}
Moreover for a collection $\mathcal{S}$ of functions on $X$, we will say that $\Lambda$ is $(D^+, \mathcal{S})$-\emph{generic} if it is $(D^+, \varphi)$-generic for every $\varphi \in \mathcal{S}$.
\newline

We are going to prove an equidistribution result for certain unbounded functions on $X$ which have `controlled' growth at infinity following~\cite{kleinbockshiweiss17}.
We will use the embedding $X \hookrightarrow \mathrm{SL}_{dk}(\mathbb{R})/\mathrm{SL}_{dk}(\mathbb{Z})$ to define a unbounded map $\alpha : X \to [0, \infty)$.
For a unimodular lattice $\Lambda$ in $\mathbb{R}^{dk}$ and a given subgroup $\Lambda' \subseteq \Lambda$, let $d(\Lambda')$ denote the covolume of $\Lambda'$ in $\operatorname{span}_\mathbb{R}(\Lambda')$ (measured with respect to the standard Euclidean structure on $\mathbb{R}^{dk}$).
Following~\cite{eskinmargulismozes98} we define
\[
    \alpha(\Lambda) := \max \{d(\Lambda')^{-1} : \Lambda' \ \text{a subgroup of} \ \Lambda\}.
\]
This maximum is attained and $\alpha$ is a proper map.
We restrict $\alpha$ to $X$.

Following~\cite{kleinbockshiweiss17}, let us denote by $C_\alpha(X)$, the space of functions $\varphi$ on $X$ satisfying the following properties:
\begin{enumerate}[label= ($C_\alpha$-\arabic*),font=\normalfont,before=\normalfont]
    \item $\varphi$ is continuous except on a set of $\mu$-measure zero;\label{cal1}
    \item The growth of $\varphi$ is majorized by $\alpha$, i.e., there exists $C > 0$ such that for all $\Lambda \in X$, we have
        \[
            |\varphi(\Lambda)| \leq C\alpha(\Lambda).
        \]\label{cal2}
\end{enumerate}

We will show that $C_\alpha(X)$ contains \emph{Siegel transforms} of Riemann integrable functions on $K_S^{d}$.
Recall that $f : K_S^{d} \to \mathbb{R}$ is called a \emph{Riemann integrable function} if $f$ is bounded with compact support and is continuous except on a set of $\lambda$-measure zero.
Define a function $\widehat{f}$ (called the \emph{Siegel transform of $f$}) on $X$ by
\begin{equation}\label{eqsiegeltransform}
    \widehat{f}(\Lambda) := \sum_{\vv{\bm{v}} \in \Lambda \smallsetminus \{0\}} f( \vv{\bm{v}}).
\end{equation}
The \emph{Siegel integral formula} (which we will prove in the next section) says that for such $f$ we have
\[
    \int_{X} \widehat{f} \, \mathrm{d}\mu = \int_{K_S^{d}} f \, \mathrm{d}\lambda.
\]

Let $\mathcal{U} = \mathcal{U}_K := u(M)$, where $u : M = \mathrm{M}_{m \times n}(K_S) \to G$ is defined as follows: for $\vartheta = (\vartheta_v)_{v \in S} \in M$
\[
    u(\vartheta) :=
    \begin{pmatrix}
        1_m & \vartheta  \\
        0 & 1_n
    \end{pmatrix}
    = {\left(
            \begin{pmatrix}
                1_{m} & \vartheta_v  \\
                0 & 1_{n}
            \end{pmatrix}
        \right)}_{v \in S}
\]
($1_\ell$ stands for the identity matrix of order $\ell$) and let $\nu$ be a Haar measure on $\mathcal{U} \cong M \cong \prod_{v \in S} K_v^{m n}$.
Let $\Lambda_0$ denote the standard lattice $\mathcal{O}^{d} \subseteq K_S^{d}$, and $\Lambda_\vartheta$ denote $u(\vartheta)\Lambda_0$.
Also fix $\Delta \in X$ and denote $u(\vartheta)\Delta$ by $\Delta_\vartheta$.
\newline

From now on elements of $K_S^{m}$ and $K_S^{n}$ will be denoted by $\vv{\bm{x}}$ and $\vv{\bm{y}}$ respectively, i.e., $\vv{\bm{x}} = (x_{1, v}, \ldots, x_{m, v})_{v \in S}$ and $\vv{\bm{y}} = (y_{1, v}, \ldots, y_{n, v})_{v \in S}$.
Whenever we say $(\vv{\bm{p}}, \vv{\bm{q}}) \in \mathcal{O}^{d}$ ($\subseteq K_S^{d} = K_S^{m} \times K_S^{n}$), we would mean that $\vv{\bm{p}} \in K_S^{m}$ and $\vv{\bm{q}} \in K_S^{n}$.
For $\vartheta = (\vartheta_v)_{v \in S} \in M$ and $(\vv{\bm{p}}, \vv{\bm{q}}) \in \mathcal{O}^{d}$, we have $\vartheta \vv{\bm{q}} \in K_S^{m}$.
A similar remark holds for $(\vv{\bm{p}}, \vv{\bm{q}}) \in \Delta$ as well.\\

We define certain `weighted quasi-norms' on $K_S^{m}$ and $K_S^{n}$:
\[
    \norm{\vv{\bm{x}}}_{\bm{a}} := \max_{\substack{1 \leq i \leq m \\ v \in S}} |x_{i, v}|^{\frac{1}{a_{i, v}}} \ {\left(\text{respectively} \ \norm{\vv{\bm{y}}}_{\bm{b}} := \max_{\substack{1 \leq j \leq n \\ v \in S}} |y_{j, v}|^{\frac{1}{b_{j, v}}}\right)}.
\]
We are now ready to state the main results proved in this paper.
Our main results are concerned with equidistribution of flows on $X$ and spiraling of weighted Diophantine inequalities in number fields.
Following the strategy of Kleinbock, Shi and Weiss, an important role in the proofs of these results is played by a counting result for solutions of Diophantine inequalities.
We begin with this result.
\subsection{Diophantine approximation and Schmidt's theorem in $K_S^{d}$}\label{ssecdiosch}
The theory of Diophantine approximation of elements in $K_S$ by rationals from $K$ has been studied by several authors.
We refer the reader to~\cite{einsiedlerghoshlytle16, kleinbockly16, anghoshguanly16, ly16, ghosh19} and the references therein.
We would like to analyze the number of solutions to the following Diophantine inequalities
\begin{equation}\label{eqdioinks}
    \begin{gathered}
        \norm{\vartheta\vv{\bm{q}} - \vv{\bm{p}}}_{\bm{a}} < \frac{c}{\norm{\vv{\bm{q}}}_{\bm{b}}}, \\
        1 \leq \norm{\vv{\bm{q}}}_{\bm{b}} < e^T,
    \end{gathered}
\end{equation}
for $\vartheta \in M, T > 0, c > 0$ and $(\vv{\bm{p}}, \vv{\bm{q}}) \in \mathcal{O}^{d}$ (or $(\vv{\bm{p}}, \vv{\bm{q}}) \in \Delta$ in general).
Accordingly, for $\vartheta \in M$, let $N_{T, c}(\vartheta)$ and $N_{T, c}(\Delta_\vartheta)$ respectively denote the number of solutions $(\vv{\bm{p}}, \vv{\bm{q}}) \in \mathcal{O}^{d}$ and $(\vv{\bm{p}}, \vv{\bm{q}}) \in \Delta$ to~\eqref{eqdioinks}.
Denote
\begin{equation*}
    \begin{gathered}
        E_{T, c} := \{(\vv{\bm{x}}, \vv{\bm{y}}) \in K_S^{m} \times K_S^{n} : \norm{\vv{\bm{x}}}_{\bm{a}} \cdot \norm{\vv{\bm{y}}}_{\bm{b}} < c, 1 \leq \norm{\vv{\bm{y}}}_{\bm{b}} < e^T\}, \\
        E_{T, c}(\Lambda) := E_{T, c} \cap \Lambda, \ \text{for all} \ \Lambda \in X
    \end{gathered}
\end{equation*}
and let $|E_{T, c}|$ be the volume of $E_{T, c}$.
It can be easily shown that $|E_{T, c}| = (2^{\#(S_\mathbb{R})} \pi^{\#(S_\mathbb{C})} \varpi^{k})^d cT$.
Note that $N_{T, c}(\vartheta) = \#(E_{T, c}(\Lambda_\vartheta))$ and $N_{T, c}(\Delta_\vartheta) = \#(E_{T, c}(\Delta_\vartheta))$. We will prove:

\begin{theorem}[A special case of Schmidt's theorem for number fields]\label{thmschmidtforks}
    For $\nu$-a.e.\ $\vartheta \in M$,
    \[
        N_{T, c}(\Delta_\vartheta) = \#(E_{T, c}(\Delta_\vartheta)) \thicksim |E_{T, c}| \ \text{as} \ T \to \infty.
    \]
    In particular,
    \[
        N_{T, c}(\vartheta) = \#(E_{T, c}(\Lambda_\vartheta)) \thicksim |E_{T, c}| \ \text{as} \ T \to \infty.
    \]
\end{theorem}

Here $f(T) \thicksim g(T)$ means that $\frac{f(T)}{g(T)} \to 1$ as $T \to \infty$.
We provide a brief history of results preceding Theorem~\ref{thmschmidtforks}.
For $K = \mathbb{Q}$, the above result, with a square root error term, and for arbitrary monotonic approximation functions, was proved in~\cite{schmidt60} by W.\ M.\ Schmidt.
This constituted a far reaching quantitative generalization of Khintchine's theorem.
By arbitrary approximating functions, we mean that Schmidt considered inequalities of the form
\[
    \norm{\vartheta \vv{\bm{q}} - \vv{\bm{p}}} < \psi(\vv{\bm{q}}).
\]

In~\cite{athreyaparrishtseng16}, Athreya, Parrish and Tseng provided an alternative proof of a simplified version of Schmidt's theorem, i.e., they provided the main term for power-type approximating functions like in~\eqref{eqdioinks}.
Their proof uses as its main tool, Birkhoff's ergodic theorem.
They did not consider Diophantine inequalities with weights, but weights can easily be incorporated into their proof.
Moving on to extensions of $\mathbb{Q}$, the only quantitative Schmidt type results that we are aware of are for imaginary quadratic extensions.
We recall Sullivan's famous paper~\cite{sullivan82} where he proved a Khintchine type theorem for $\mathbb{Q}(\sqrt{d})$, for $d$ a square-free negative integer.
Subsequently, a quantitative version of Sullivan's theorem was established by Nakada~\cite{nakada88}, namely he established Theorem~\ref{thmschmidtforks} for the case $K = \mathbb{Q}(\sqrt{d})$.
As far as we are aware, Theorem~\ref{thmschmidtforks} is new in all other cases.
It constitutes a quantitative version of Khintchine's theorem for arbitrary number fields.
We note that Khintchine's theorem for number fields is proved in T.\ Ly's thesis~\cite{ly16}.
It would be interesting to obtain an error bound in the context of Theorem \ref{thmschmidtforks}. We have followed the approach of~\cite{athreyaparrishtseng16} which does not yield an error term.
\newline

Before proving Theorem~\ref{thmschmidtforks}, we will prove the following result concerning unimodular lattices in $X$.

\begin{theorem}\label{thmschmidtforkslattice}
    For $\mu$-a.e.\ $\Lambda \in X$
    \[
        \#{\left(E_{T, c}(\Lambda)\right)} \thicksim |E_{T, c}| \ \text{as} \ T \to \infty.
    \]
\end{theorem}

\subsection{Spiraling of approximations and spherical averages}\label{}
Let $F_{\bm{a}t}$ be the $\bm{a}$-weighted flow on $K_S^{m}$ defined by
\[
    F_{\bm{a}t}(\vv{\bm{x}}) := (e^{a_{1, v} t}x_{1, v}, \ldots, e^{a_{m, v} t}x_{m, v})_{v \in S} \ \text{for} \ \vv{\bm{x}} \in K_S^{m} \ \text{and} \ t \in \mathbb{R}
\]
and define $F_{\bm{b}t}$ on $K_S^{n}$ similarly.
Define
\[
    \mathbb{S}^{km - 1} := {\left\{\vv{\bm{x}} \in K_S^{m} : \sum_{v \in S} \sum_{i = 1}^{m} |x_{i, v}|^2 = 1\right\}}
\]
and define $\mathbb{S}^{kn - 1}$ similarly.
For nonzero $\vv{\bm{x}} \in K_S^{m}$ and $\vv{\bm{y}} \in K_S^{n}$ let
\begin{equation}\label{eqprojforksj}
    \pi_{\bm{a}}(\vv{\bm{x}}) := \{F_{\bm{a}t}(\vv{\bm{x}}) : t \in \mathbb{R}\} \cap \mathbb{S}^{km-1}, \ \pi_{\bm{b}}(\vv{\bm{y}}) := \{F_{\bm{b}t}(\vv{\bm{y}}) : t \in \mathbb{R}\} \cap \mathbb{S}^{kn-1}
\end{equation}
(these intersections are clearly singleton).

For $\vartheta \in M$ and $A \subseteq \mathbb{S}^{km-1}, B \subseteq \mathbb{S}^{kn-1}$ with boundaries of measure zero, let $N_{T, c}(A, B; \vartheta)$ and $N_{T, c}(A, B; \Delta_\vartheta)$ respectively denote the number of solutions $(\vv{\bm{p}}, \vv{\bm{q}}) \in \mathcal{O}^{d}$ and $(\vv{\bm{p}}, \vv{\bm{q}}) \in \Delta$ to
\begin{equation}\label{eqdioinksAB}
    \begin{gathered}
        \norm{\vartheta\vv{\bm{q}} - \vv{\bm{p}}}_{\bm{a}} < \frac{c}{\norm{\vv{\bm{q}}}_{\bm{b}}}, \\
        1 \leq \norm{\vv{\bm{q}}}_{\bm{b}} < e^T, \\
        \pi_{\bm{a}}(\vartheta\vv{\bm{q}} - \vv{\bm{p}}) \in A \ \text{and} \ \pi_{\bm{b}}(\vv{\bm{q}}) \in B.
    \end{gathered}
\end{equation}
Define
\begin{gather*}
    E_{T, c}(A, B) := {\left\{(\vv{\bm{x}}, \vv{\bm{y}}) \in K_S^{m} \times K_S^{n} : \substack{\displaystyle\norm{\vv{\bm{x}}}_{\bm{a}} \cdot \norm{\vv{\bm{y}}}_{\bm{b}} < c, 1 \leq \norm{\vv{\bm{y}}}_{\bm{b}} < e^T, \\ \\ \displaystyle\pi_{\bm{a}}(\vv{\bm{x}}) \in A, \pi_{\bm{b}}(\vv{\bm{y}}) \in B}\right\}}, \\
    E_{T, c}(A, B; \Lambda) := E_{T, c}(A, B) \cap \Lambda, \ \text{for all} \ \Lambda \in X.
\end{gather*}
It can be shown that $|E_{T, c}(A, B)| = |E_{T, c}| \vol(A) \vol(B)$, where $\vol$ denotes the standard probability measure on the spheres $\mathbb{S}^{km-1}$ and $\mathbb{S}^{kn-1}$.
Note that $N_{T, c}(A, B; \vartheta) = \#(E_{T, c}(A, B; \Lambda_\vartheta))$ and $N_{T, c}(A, B; \Delta_\vartheta) = \#(E_{T, c}(A, B; \Delta_\vartheta))$.
We prove that

\begin{theorem}\label{thmcalgenofDevth}
    Let $\Delta \in X$.
    Then for almost every $\vartheta \in M$, $\Delta_\vartheta$ is $(D^+, C_\alpha(X))$-generic.
\end{theorem}

\noindent As mentioned previously, this result (for $\mathbb{Q}$) answers a question of Kleinbock, Shi and Weiss~\cite{kleinbockshiweiss17} (cf.~\cite{kleinbockshiweiss17} page 861, following Theorem 1.3).
Finally, we have the weighted, number field, spiraling theorem.

\begin{theorem}\label{thmequidistonsphereforksk}
    Let $c > 0$ and measurable subsets $A \subseteq \mathbb{S}^{km-1}, B \subseteq \mathbb{S}^{kn-1}$ with boundaries of measure zero be given.
    Then for a.e.\ $\vartheta \in M$, as $T \to \infty$, the number of solutions $(\vv{\bm{p}}, \vv{\bm{q}}) \in \Delta$ to~\eqref{eqdioinksAB} has the same asymptotic growth as the volume of the set $E_{T, c}(A, B)$, i.e.,
    \[
        N_{T, c}(A, B; \Delta_\vartheta) = \#(E_{T, c}(A, B; \Delta_\vartheta)) \thicksim |E_{T, c}(A, B)| \ \text{as} \ T \to \infty.
    \]
    In particular
    \[
        N_{T, c}(A, B; \vartheta) = \#(E_{T, c}(A, B; \Lambda_\vartheta)) \thicksim |E_{T, c}(A, B)| \ \text{as} \ T \to \infty.
    \]
\end{theorem}

Note that the distribution of approximates from a different, probabilistic point of view has been undertaken in~\cite{athreyaghosh18}.
It would be interesting to investigate these in the number field context.
Moreover, these questions can also be investigated in the function field context (cf.~\cite{athreyaghoshprasad12, ghoshroyals15, kwonlim18}).


\section{Proof of Theorem~\ref{thmschmidtforkslattice}}\label{secprfschkslat}

We will make use of Siegel's mean value theorem for number fields, which is a theorem about the average number of lattice points in given subset of $K_S^{d}$.
The result below follows from a more general result due to A.\ Weil~\cite{weil46}.

\begin{theorem}[Weil \cite{weil46}]\label{thmsmvtforks}
    For $f \in L^1(K_S^{d}, \lambda)$,
    \begin{equation}\label{eqsmvtforks}
        \int_{X}\widehat{f} \, \mathrm{d}\mu = \int_{K_S^{d}} f \, \mathrm{d}\lambda.
    \end{equation}
\end{theorem}

Following~\cite{kleinbockshiweiss17}, we now deduce some estimates which we will use in the proof of Theorems~\ref{thmschmidtforks} --~\ref{thmequidistonsphereforksk}.
For $\vv{\bm{z}} = (\vv{\bm{x}}, \vv{\bm{y}}) \in K_S^{m} \times K_S^{n} = K_S^{d}$, write $g_t\vv{\bm{z}} = (\vv{\bm{x}}_t, \vv{\bm{y}}_t)$.
Then we have
\[
    \norm{\vv{\bm{x}}_t}_{\bm{a}} = e^t \norm{\vv{\bm{x}}}_{\bm{a}} \ \text{and} \ \norm{\vv{\bm{y}}_t}_{\bm{b}} = e^{-t} \norm{\vv{\bm{y}}}_{\bm{b}}.
\]
Let $f_{r, c} = \mathbbm{1}_{E_{r, c}}$ be the characteristic function of $E_{r, c}$, then we have
\[
    f_{r, c}(g_t \vv{\bm{z}}) =
    \begin{dcases}
        1 & \text{if} \ e^t \leq \norm{\vv{\bm{y}}}_{\bm{b}} < e^{t + r} \\
        0 & \text{otherwise.} \ 
    \end{dcases}
\]
Therefore
\begin{gather*}
    \vv{\bm{z}} \in E_{T, c} \implies |\{t \in [0, T] : g_t \vv{\bm{z}} \in E_{r, c}\}| \leq r, \\
    \vv{\bm{z}} \in E_{T, c} \smallsetminus E_{r, c} \implies |\{t \in [0, T] : g_t \vv{\bm{z}} \in E_{r, c}\}| =  r, \\
\end{gather*}
and
\begin{gather*}
    g_t \vv{\bm{z}} \in E_{r, c} \ \text{for some} \ t \in [0, T] \implies \vv{\bm{z}} \in E_{T+r, c}.
\end{gather*}
Using~\eqref{eqsiegeltransform} and changing the order of summation and integration it follows that for any $\Lambda \in X$ and any $T > r$
\begin{equation}\label{eqsandwichequidistfhat1}
    \#(\Lambda \cap (E_{T, c} \smallsetminus E_{r, c})) \leq \frac{1}{r} \int_{0}^{T} \widehat{f}_{r, c}(g_t \Lambda) \, \mathrm{d}t \leq \#(\Lambda \cap E_{T+r, c}).
\end{equation}
From~\eqref{eqsandwichequidistfhat1} it follows that
\begin{equation}\label{eqsandwichequidistfhat2}
    \frac{1}{r} \int_{0}^{T-r} \widehat{f}_{r, c}(g_t \Lambda) \, \mathrm{d}t \leq \#(E_{T, c}(\Lambda)) \leq \frac{1}{r} \int_{0}^{T} \widehat{f}_{r, c}(g_t \Lambda) \, \mathrm{d}t + \#(E_{r, c}(\Lambda)).
\end{equation}

Similarly, let $f_{A, B, r, c} = \mathbbm{1}_{E_{r, c}(A, B)}$ be the characteristic function of $E_{r, c}(A, B)$.
Then by similar calculations we get that for any $\Lambda \in X$
\begin{equation}\label{eqsandwichequidistfhat3}
    \frac{1}{r} \int_{0}^{T-r} \widehat{f}_{A, B, r, c}(g_t \Lambda) \, \mathrm{d}t \leq \#(E_{T, c}(A, B; \Lambda)) \leq \frac{1}{r} \int_{0}^{T} \widehat{f}_{A, B, r, c}(g_t \Lambda) \, \mathrm{d}t + \#(E_{r, c}(A, B; \Lambda)).
\end{equation}

By Moore's ergodicity theorem, the action of $\{g_t\}$ on $X$ is ergodic with respect to the Haar measure $\mu$. Further Theorem~\ref{thmsmvtforks} (Siegel's mean value theorem), implies that $\widehat{f}_{r, c} \in L^1(X)$.
We may therefore apply Birkhoff's ergodic theorem (stated below) to $\widehat{f}_{r, c}$.
\begin{theorem}[Birkhoff ergodic theorem]\label{thmbirkhoffforks}
    Let ${\{g_t\}}_{t>0}$ be an ergodic measure-preserving action on a probability space $(\Omega, \mu)$ and $f \in L^1(\Omega)$.
    Then for almost every $x \in \Omega$, we have that
    \[
        \lim_{T \to \infty} \frac{1}{T} \int_{0}^{T} f(g_t x) \, \mathrm{d}t = \int_{\Omega} f \, \mathrm{d}\mu.
    \]
\end{theorem}

\begin{uremark}
    Using ergodicity of $\{g_t\}$ and the Birkhoff ergodic theorem, we get that for a given $L^1(X)$ function $\varphi$, $\mu$-almost every $\Lambda \in X$ is Birkhoff generic with respect to $\varphi$.
\end{uremark}

Applying Theorem~\ref{thmbirkhoffforks} to~\eqref{eqsandwichequidistfhat2} and using Theorem~\ref{thmsmvtforks}, we see that for almost every $\Lambda \in X$
\begin{equation}\label{eqbirgenofaela}
    \lim_{T \to \infty} \frac{\#(E_{T, c}(\Lambda))}{T} = \lim_{T \to \infty} \frac{1}{Tr} \int_{0}^{T} \widehat{f}_{r, c}(g_t \Lambda) \, \mathrm{d}t = \frac{1}{r} \int_{X} \widehat{f}_{r, c} \, \mathrm{d}\mu = \frac{1}{r} \int_{K_S^{d}} f_{r, c} \, \mathrm{d}\lambda = \frac{1}{r} |E_{r, c}|.
\end{equation}
The volume estimate from \S\ref{ssecdiosch} implies that $r \cdot |E_{T, c}| =  T \cdot |E_{r, c}|$, we have proved Theorem~\ref{thmschmidtforkslattice}.

\section{Proof of Theorem~\ref{thmschmidtforks}}\label{}

We now wish to deduce Theorem~\ref{thmschmidtforks} from Theorem~\ref{thmschmidtforkslattice}.
Define
\[
    \mathcal{M} = {\left\{
            \begin{pmatrix}
                \alpha & \delta  \\
                \beta & \gamma
            \end{pmatrix}
            \in G : \alpha \in \mathrm{GL}_{m}(K_S), \beta \in \mathrm{M}_{n \times m}(K_S), \gamma \in \mathrm{M}_{n}(K_S), \delta \in M\right\}}
\]
Note that $\mathcal{M}$ is open and $G \smallsetminus \mathcal{M}$ has $\mu$-measure zero.
Let
\[
    \mathcal{H} = {\left\{
            \begin{pmatrix}
                \alpha & 0  \\
                \beta & \gamma
            \end{pmatrix}
            \in G : \alpha \in \mathrm{GL}_{m}(K_S), \beta \in \mathrm{M}_{n \times m}(K_S), \gamma \in \mathrm{M}_{n}(K_S)\right\}},
\]
which is a subgroup of $G$.
Let $\nu_{\mathcal{H}}$ denote the left Haar measure on $\mathcal{H}$.

\begin{lemma}\label{lemprodmeasonslkks}
    The map
    \begin{gather*}
        \mathcal{H} \times \mathcal{U} \to \mathcal{M} \\
        {\left(
                \begin{pmatrix}
                    \alpha & 0  \\
                    \beta & \gamma
                \end{pmatrix}
                , u(\vartheta)\right)} \mapsto
        \begin{pmatrix}
            \alpha & 0  \\
            \beta & \gamma
        \end{pmatrix}
        u(\vartheta)
    \end{gather*}
    is a homeomorphism.
\end{lemma}
\begin{proof}
    For $
    \begin{pmatrix}
        \alpha' & \delta'  \\
        \beta' & \gamma'
    \end{pmatrix}
    \in \mathcal{M}$, let $\alpha = \alpha', \beta = \beta', \vartheta = \alpha'^{-1}\delta'$, and $\gamma = \gamma' - \beta'\alpha'^{-1}\delta'$.
    Then
    \[
        \begin{pmatrix}
            \alpha & 0  \\
            \beta & \gamma
        \end{pmatrix}
        \begin{pmatrix}
            1 & \vartheta  \\
            0 & 1
        \end{pmatrix}
        =
        \begin{pmatrix}
            \alpha' & \delta'  \\
            \beta' & \gamma'
        \end{pmatrix}.
    \]
    Since
    \[
        \det (
        \begin{pmatrix}
            \alpha' & \delta'  \\
            \beta' & \gamma'
        \end{pmatrix}
        \begin{pmatrix}
            1 & \vartheta  \\
            0 & 1
        \end{pmatrix}
        ^{-1}) = 1,
    \]
    we have $\det(\gamma) \neq 0$.

    If an element has two such decompositions, then we have
    \[
        \begin{pmatrix}
            \alpha_1 & 0  \\
            \beta_1 & \gamma_1
        \end{pmatrix}
        ^{-1}
        \begin{pmatrix}
            \alpha_2 & 0  \\
            \beta_2 & \gamma_2
        \end{pmatrix}
        =
        u(\vartheta_1)u(\vartheta_2)^{-1}.
    \]
    Multiplying the matrices we get that $\alpha_1 = \alpha_2, \beta_1 = \beta_2, \gamma_1 = \gamma_2$ and $\vartheta_1 = \vartheta_2$.
\end{proof}

The parametrization from Lemma~\ref{lemprodmeasonslkks} gives us a Haar measure $\nu_{\mathcal{H}} \times \nu$ on $G$, thus it must be a constant multiple of $\mu$.
By normalizing $\nu_{\mathcal{H}}$ appropriately, we may assume that the constant is 1.

Given a subset $B \subseteq G$, define $B^\Delta := \{g\Delta : g \in B\}$.
Let $u(\vartheta') \in \mathcal{U}$, $W_{\mathcal{H}}$ be a open neighborhood in $\mathcal{H}$ around the identity and $W_{\mathcal{U}}$ be a open neighborhood in $\mathcal{U}$ around $u(\vartheta')$ so that $W := W_{\mathcal{H}}W_{\mathcal{U}}$ is a open neighborhood of $u(\vartheta')$ in $G$.

\begin{proposition}\label{propfubinistyle}
    Let $W$ be as in the previous paragraph.
    Then for $\nu_{\mathcal{H}}$-almost every $h \in W_{\mathcal{H}}$, there exists a measurable subset $V_h \subseteq W_{\mathcal{U}}$ such that $\nu(V_h) = \nu(W_{\mathcal{U}})$ and for every $u(\vartheta) \in V_h$, the lattice $h u(\vartheta) \Delta = h \Delta_\vartheta$ is Birkhoff generic with respect to the function $\widehat{f}_{r, c}$.
\end{proposition}
\begin{proof}
    Since $\mu$-almost every element in $X$ is Birkhoff generic with respect to $\widehat{f}_{r, c}$, there exists a set $W_{bg} \subseteq W$ such that every element in $W_{bg}^\Delta$ is Birkhoff generic with respect to the function $\widehat{f}_{r, c}$ and such that $\mu(W_{bg}) = \mu(W)$.

    For $h \in W_\mathcal{H}$, define $W_{bg, h} := \{u(\vartheta) \in W_{\mathcal{U}} : hu(\vartheta) \in W_{bg}\}$.
    Then Fubini's theorem implies that $W_{bg, h}$ is measurable for almost all $h \in W_{\mathcal{H}}$.
    We claim that $V_h = W_{bg, h}$ works, i.e., $\nu(W_{bg, h}) = \nu(W_{\mathcal{U}})$ for almost every $h \in W_\mathcal{H}$.

    If not, there exists a subset $W_{\mathcal{H}}'$ of $W_{\mathcal{H}}$ of positive $\nu_{\mathcal{H}}$-measure such that, for every element $h \in W_{\mathcal{H}}'$, we have $\nu(W_{bg, h}) < \nu(W_{\mathcal{U}})$.
    Integrating using Fubini's theorem, we have
    \begin{align*}
        \mu(W_{bg}) &= \int_{W_{\mathcal{H}}} \int_{W_{\mathcal{U}}} \mathbbm{1}_{W_{bg}}(hu(\vartheta)) \, \mathrm{d}\nu(\vartheta) \, \mathrm{d}\nu_{\mathcal{H}}(h) \\
        &= \int_{W_{\mathcal{H}}} \int_{W_{\mathcal{U}}} \mathbbm{1}_{W_{\mathcal{H}}}(h) \cdot \mathbbm{1}_{W_{bg, h}}(u(\vartheta)) \, \mathrm{d}\nu(\vartheta) \, \mathrm{d}\nu_{\mathcal{H}}(h) \\
        &= \int_{W_{\mathcal{H}} \smallsetminus W_{\mathcal{H}}'} \int_{W_{\mathcal{U}}} \mathbbm{1}_{W_{\mathcal{H}} \smallsetminus W_{\mathcal{H}}'}(h) \cdot \mathbbm{1}_{W_{bg, h}}(u(\vartheta)) \, \mathrm{d}\nu(\vartheta) \, \mathrm{d}\nu_{\mathcal{H}}(h) \\
        & \hspace{4em} + \int_{W_{\mathcal{H}}'} \int_{W_{\mathcal{U}}} \mathbbm{1}_{W_{\mathcal{H}}'}(h) \cdot \mathbbm{1}_{W_{bg, h}}(u(\vartheta)) \, \mathrm{d}\nu(\vartheta) \, \mathrm{d}\nu_{\mathcal{H}}(h) \\
        &< \int_{W_{\mathcal{H}} \smallsetminus W_{\mathcal{H}}'} \int_{W_{\mathcal{U}}} \mathbbm{1}_{W_{\mathcal{H}} \smallsetminus W_{\mathcal{H}}'}(h) \cdot \mathbbm{1}_{W_{\mathcal{U}}}(u(\vartheta)) \, \mathrm{d}\nu(\vartheta) \, \mathrm{d}\nu_{\mathcal{H}}(h) \\
        & \hspace{4em} + \int_{W_{\mathcal{H}}'} \int_{W_{\mathcal{U}}} \mathbbm{1}_{W_{\mathcal{H}}'}(h) \cdot \mathbbm{1}_{W_{\mathcal{U}}}(u(\vartheta)) \, \mathrm{d}\nu(\vartheta) \, \mathrm{d}\nu_{\mathcal{H}}(h) \\
        &= \int_{W_{\mathcal{H}}} \int_{W_{\mathcal{U}}} \mathbbm{1}_{W_{\mathcal{H}}}(h) \cdot \mathbbm{1}_{W_{\mathcal{U}}}(u(\vartheta)) \, \mathrm{d}\nu(\vartheta) \, \mathrm{d}\nu_{\mathcal{H}}(h) \\
        &= \int_{W_{\mathcal{H}}} \int_{W_{\mathcal{U}}} \mathbbm{1}_{W_{\mathcal{H}}W_{\mathcal{U}}}(hu(\vartheta)) \, \mathrm{d}\nu(\vartheta) \, \mathrm{d}\nu_{\mathcal{H}}(h) \\
        &= \mu(W),
    \end{align*}
    a contradiction.
\end{proof}

In the notation as in Proposition~\ref{propfubinistyle}, let $V = \{h \in W_{\mathcal{H}} : \nu(V_h) = \nu(W_{\mathcal{U}})\}$.
Then $\nu_{\mathcal{H}}(V) = \nu_{\mathcal{H}}(W_{\mathcal{H}})$.
In the next stage we are going to approximate $\Delta_\vartheta$ (for $u(\vartheta) \in W_{\mathcal{U}}$) by certain sequence of Birkhoff generic points $\{\Delta_\ell\}_{\ell \in \mathbb{N}}$ in $W_{bg}^\Delta$ in such a way that the Birkhoff genericity of $\Delta_\ell$ implies the Birkhoff genericity of $\Delta_\vartheta$ for almost all $u(\vartheta) \in W_{\mathcal{U}}$.
\newline

Let $\{\varepsilon_\ell\}_{\ell \in \mathbb{N}} \to 0$ be a sequence of positive reals.
For each $\ell \in \mathbb{N}$ we are going to choose $h_\ell =
\begin{pmatrix}
    \alpha_\ell & 0  \\
    \beta_\ell & \gamma_\ell
\end{pmatrix}
$ from $V$ satisfying the following two conditions:
\renewcommand{\labelenumi}{(\roman{enumi})}
\begin{enumerate}
    \item $h_\ell \to 1_d$ as $\ell \to \infty$.
        For each $h_\ell \ \exists~V_\ell \subseteq W_{\mathcal{U}}$ such that $\nu(V_\ell) = \nu(W_{\mathcal{U}})$ and for all $u(\vartheta) \in V_\ell$ the lattice $h_\ell \Delta_\vartheta$ is Birkhoff generic with respect to $\widehat{f}_{c-\varepsilon_\ell}$ and $\widehat{f}_{c + \varepsilon_\ell}$.

        Let $V_\infty := \bigcap V_\ell$, then $\nu(V_\infty) = \nu(W_{\mathcal{U}})$.
        Fix $u(\vartheta) \in V_\infty$, then $\Delta_\ell = h_\ell \Delta_\vartheta$ is Birkhoff generic with respect to $\widehat{f}_{r, c - \varepsilon_\ell}$ and $\widehat{f}_{r, c + \varepsilon_\ell}$ for all $\ell \in \mathbb{N}$.

        Now we are going to choose the `speed' at which $h_\ell \to 1_d$.

    \item Let $E_c := \{(\vv{\bm{x}}, \vv{\bm{y}}) \in K_S^{m} \times K_S^{n} : \norm{\vv{\bm{x}}}_{\bm{a}} \cdot \norm{\vv{\bm{y}}}_{\bm{b}} < c, 1 \leq \norm{\vv{\bm{y}}}_{\bm{b}}\}$.
        Then $E_c \cap \Delta_\vartheta$ and $h_\ell E_c \cap \Delta_\ell$ naturally correspond to each other.
        We have
        \[
            h_\ell E_c = \{(\alpha_\ell \vv{\bm{x}}, \beta_\ell \vv{\bm{x}} + \gamma_\ell \vv{\bm{y}}) : (\vv{\bm{x}}, \vv{\bm{y}}) \in E_c\}.
        \]

        Since $\vv{\bm{x}}$ is uniformly bounded for $(\vv{\bm{x}}, \vv{\bm{y}}) \in E_c$, we can choose $h_\ell$ so close to $1_d$ that
        \begin{subequations}
            \begin{gather}
                (1 - \varepsilon_\ell')\norm{\vv{\bm{x}}}_{\bm{a}} \cdot \norm{\vv{\bm{y}}}_{\bm{b}} \leq \norm{\alpha_\ell \vv{\bm{x}}}_{\bm{a}} \cdot \norm{\beta_\ell \vv{\bm{x}} + \gamma_\ell \vv{\bm{y}}}_{\bm{b}} \leq (1 + \varepsilon_\ell')\norm{\vv{\bm{x}}}_{\bm{a}} \cdot \norm{\vv{\bm{y}}}_{\bm{b}} \label{eqclosehla}\\
                (1 - \varepsilon_\ell'')\norm{\vv{\bm{y}}}_{\bm{b}} - \varepsilon_\ell''' \leq \norm{\beta_\ell \vv{\bm{x}} + \gamma_\ell \vv{\bm{y}}}_{\bm{b}} \leq (1 + \varepsilon_\ell'')\norm{\vv{\bm{y}}}_{\bm{b}} + \varepsilon_\ell''' \label{eqclosehlb}
            \end{gather}
        \end{subequations}
        where $\varepsilon_\ell', \varepsilon_\ell''$ and $\varepsilon_\ell''' \to 0$ are positive real numbers satisfying $(1 + \varepsilon_\ell')c < c + \varepsilon_\ell$,  $\frac{c - \varepsilon_\ell}{1 - \varepsilon_\ell'} < c$ and $\varepsilon_\ell'', \varepsilon_\ell''' < 1/3$.
\end{enumerate}
We have thus chosen $h_\ell$.

Note that~\eqref{eqclosehla} implies that $h_\ell E_c$ can be approximated from inside by $E_{c - \varepsilon_\ell}$ and from outside by $E_{c + \varepsilon_\ell}$, possibly up to two precompact sets $\mathcal{C}_\ell^1$ and $\mathcal{C}_\ell^2$.
The set $\mathcal{C}_\ell^2$ appears as follows:
For $h_\ell E_c$ there might exist points $(\vv{\bm{x}}, \vv{\bm{y}}) \in E_c$ such that $\norm{\beta_\ell \vv{\bm{x}} + \gamma_\ell \vv{\bm{y}}}_{\bm{b}} < 1$.
We need to exclude these points from $E_c$.
Let $S := \{(\vv{\bm{x}}, \vv{\bm{y}}) \in K_S^{m} \times K_S^{n} : \norm{\vv{\bm{y}}}_{\bm{b}} \leq 1\}$ and
\begin{align*}
    \mathcal{C}_\ell^2 &:= \{(\vv{\bm{x}}, \vv{\bm{y}}) \in E_c : h_\ell(\vv{\bm{x}}, \vv{\bm{y}}) \in S\} \\
    &= h_\ell^{-1} (S) \cap E_c.
\end{align*}
Then~\eqref{eqclosehlb} implies that $\mathcal{C}_\ell^2$ is bounded, and we have
\[
    h_\ell(E_c \smallsetminus \mathcal{C}_\ell^2) \subseteq E_{c + \varepsilon_\ell}.
\]
Similarly let
\begin{align*}
    \mathcal{C}_\ell^1 &:= \{(\vv{\bm{x}}, \vv{\bm{y}}) \in E_{c - \varepsilon_\ell} : h_\ell^{-1}(\vv{\bm{x}}, \vv{\bm{y}}) \in S \cup \mathcal{C}_\ell^2\} \\
    &\subseteq h_\ell(\overline{\mathcal{C}_\ell^2}) \cup (h_\ell(S) \cap E_{c - \varepsilon_\ell}).
\end{align*}
Again~\eqref{eqclosehlb} implies $\mathcal{C}_\ell^1$ is bounded, and we have
\[
    E_{c - \varepsilon_\ell} \smallsetminus \mathcal{C}_\ell^1 \subseteq h_\ell(E_c \smallsetminus \mathcal{C}_\ell^2) \subseteq E_{c + \varepsilon_\ell}.
\]

By similar arguments and using~\eqref{eqclosehlb}, we see that
\[
    E_{T-1, c-\varepsilon_\ell} \smallsetminus \mathcal{C}_\ell^1 \subseteq h_\ell(E_{T, c} \smallsetminus \mathcal{C}_\ell^2) \subseteq E_{T+1, c+\varepsilon_\ell}
\]
for all $\ell$ and $T > 1$.

Therefore
\begin{equation}\label{eqimpineqforuvthgeni}
    \#(E_{T-1, c - \varepsilon_\ell}(\Delta_\ell)) - C_\ell^1 \leq \#(E_{T, c}(\Delta_\vartheta)) - C_\ell^2 \leq \#(E_{T+1, c+\varepsilon_\ell}(\Delta_\ell))
\end{equation}
where $C_\ell^1 := \#(\mathcal{C}_\ell^1 \cap \Delta_\ell), C_\ell^2 := \#(\mathcal{C}_\ell^2 \cap \Delta_\vartheta)$.
Since a precompact set of $K_S^{d}$ can only have a finite number of lattice points, it follows that $C_\ell^1, C_\ell^2 < \infty$.
Consequently, from~\eqref{eqimpineqforuvthgeni} we have that
\[
    \lim_{T \to \infty} \frac{\#(E_{T-1, c-\varepsilon_\ell}(\Delta_\ell))}{T-1} \leq \lim_{T \to \infty} \frac{\#(E_{T, c}(\Delta_\vartheta))}{T} \leq \lim_{T \to \infty} \frac{\#(E_{T+1, c+\varepsilon_\ell}(\Delta_\ell))}{T+1}.
\]
Since $\Delta_\ell$ is Birkhoff generic with respect to $\widehat{f}_{r, c-\varepsilon_\ell'}, \widehat{f}_{r, c+\varepsilon_\ell'}$ for all $\ell$, using~\eqref{eqbirgenofaela} we have that
\[
    \frac{1}{r} |E_{r, c-\varepsilon_\ell}| \leq \lim_{T \to \infty} \frac{\#(E_{T, c}(\Delta_\vartheta)}{T} \leq \frac{1}{r} |E_{r, c+\varepsilon_\ell}|.
\]
Using the volume estimate from \S\ref{ssecdiosch} and letting $\ell \to \infty$ we get that for all $u(\vartheta) \in V_\infty$,
\[
    \lim_{T \to \infty} \frac{\#(E_{T, c}(\Delta_\vartheta))}{T} = \frac{1}{r} |E_{r, c}|.
\]

Since $u(\vartheta') \in \mathcal{U}$ is arbitrary, we have that for $\nu$-almost all $u(\vartheta) \in \mathcal{U}$,
\[
    \lim_{T \to \infty} \frac{\#(E_{T, c}(\Delta_\vartheta))}{T} = \frac{1}{r} |E_{r, c}|.
\]
Since the volume estimate from \S\ref{ssecdiosch} implies that $r \cdot |E_{T, c}| = T \cdot |E_{r, c}|$, we have proved Theorem~\ref{thmschmidtforks}.

\section{Proof of Theorems~\ref{thmcalgenofDevth} and~\ref{thmequidistonsphereforksk}}\label{}

Let us prove some results using Theorem~\ref{thmschmidtforks}.

\begin{corollary}\label{corofschi}
    Let $r, c > 0$.
    Then for almost every $\vartheta \in M$,
    \begin{equation}\label{eqEequidist}
        \lim_{T \to \infty} \frac{1}{T} \int_{0}^{T} \widehat{f}_{r, c}(g_t \Delta_\vartheta) \, \mathrm{d}t = |E_{r, c}|.
    \end{equation}
\end{corollary}
\begin{proof}
    This result follows from~\eqref{eqsandwichequidistfhat1}, Theorem~\ref{thmschmidtforks} and the volume estimate for $E_{T, c}$ from \S\ref{ssecdiosch}.
\end{proof}

\begin{corollary}\label{corofschii}
    For $r, c > 0$, let
    \[
        F_{r, c} := \{(\vv{\bm{x}}, \vv{\bm{y}}) \in K_S^{m} \times K_S^{n} : \norm{\vv{\bm{x}}}_{\bm{a}} \cdot \norm{\vv{\bm{y}}}_{\bm{b}} < c, 1 \leq \norm{\vv{\bm{x}}}_{\bm{a}} < e^r\}.
    \]
    Then for almost every $\vartheta \in M$,
    \begin{equation}\label{eqFequidist}
        \lim_{T \to \infty} \frac{1}{T} \int_{0}^{T} \widehat{\mathbbm{1}}_{F_{r, c}}(g_t \Delta_\vartheta) \, \mathrm{d}t = |F_{r, c}|.
    \end{equation}
\end{corollary}
\begin{proof}
    For every $T > |\ln c|$ and $\vartheta \in M$, similar to~\eqref{eqsandwichequidistfhat1}, we have
    \[
        \frac{1}{r} \int_{0}^{T} \widehat{\mathbbm{1}}_{F_{r, c}} (g_t \Delta_\vartheta) \, \mathrm{d}t \leq \#(E_{T + \ln c, c}(\Delta_\vartheta)) + \#(\widetilde{F} \cap \Delta_\vartheta)
    \]
    where $\widetilde{F} := \{(\vv{\bm{x}}, \vv{\bm{y}}) \in K_S^{m} \times K_S^{n} : \norm{\vv{\bm{x}}}_{\bm{a}} < e^r, \norm{\vv{\bm{y}}}_{\bm{b}} < c\}$.
    For every $\vartheta$, $\#(\widetilde{F} \cap \Delta_\vartheta)$ is a number independent of $T$ and $|F_{r, c}| = |E_{r, c}|$.
    So for $\vartheta \in M$ satisfying~\eqref{eqEequidist} we have
    \begin{equation}\label{eqFequidistls}
        \limsup_{T \to \infty} \frac{1}{T} \int_{0}^{T} \widehat{\mathbbm{1}}_{F_{r, c}}(g_t \Delta_\vartheta) \, \mathrm{d}t \leq |F_{r, c}|.
    \end{equation}

    For $r_1 < r_2$ and $0 \leq c_1 < c_2$ define
    \[
        E_{r_1, r_2; c_1, c_2} := \{(\vv{\bm{x}}, \vv{\bm{y}}) \in K_S^{m} \times K_S^{n} : c_1 \leq \norm{\vv{\bm{x}}}_{\bm{a}} \cdot \norm{\vv{\bm{y}}}_{\bm{b}} < c_2, e^{r_1} \leq \norm{\vv{\bm{y}}} < e^{r_2}\}.
    \]
    Consider
    \[
        \frac{1}{T} \int_{0}^{T} \widehat{\mathbbm{1}}_{E_{r_2 - r_1, c}}(g_{r_1}g_t \Delta_\vartheta) \, \mathrm{d}t
    \]
    On the one hand it is equal to
    \begin{align*}
        & \quad \frac{1}{T} \int_{r_1}^{T+r_1} \widehat{\mathbbm{1}}_{E_{r_2 - r_1}, c}(g_t \Delta_\vartheta) \, \mathrm{d}t \\
        &= \frac{T+r_1}{T} \cdot \frac{1}{T+r_1} \int_{0}^{T+r_1} \widehat{\mathbbm{1}}_{E_{r_2 - r_1, c}}(g_t \Delta_\vartheta) \, \mathrm{d}t - \frac{1}{T} \int_{0}^{r_1} \widehat{\mathbbm{1}}_{E_{r_2 - r_1, c}}(g_t \Delta_\vartheta) \, \mathrm{d}t.
    \end{align*}
    On the other hand, it is equal to
    \begin{align*}
        & \quad \frac{1}{T} \int_{0}^{T} \widehat{\mathbbm{1}}_{g_{-r_1}E_{r_2 - r_1, c}}(g_t \Delta_\vartheta) \, \mathrm{d}t \\
        &= \frac{1}{T} \int_{0}^{T} \widehat{\mathbbm{1}}_{E_{r_1, r_2; 0, c}}(g_t \Delta_\vartheta) \, \mathrm{d}t
    \end{align*}
    since $g_{-r_1}E_{r_2 - r_1, c} = E_{r_1, r_2; 0, c}$.
    Therefore, for almost every $\vartheta \in M$
    \[
        \lim_{T \to \infty} \frac{1}{T} \int_{0}^{T} \widehat{\mathbbm{1}}_{E_{r_1, r_2; 0, c}}(g_t \Delta_\vartheta) \, \mathrm{d}t = |E_{r_1, r_2; 0, c}|
    \]
    again using $g_{-r_1}E_{r_2 - r_1, c} = E_{r_1, r_2; 0, c}$ and the fact that $\{g_t\}$ is measure preserving.

    For $0 \leq c_1 < c_2$ note that $E_{r_1, r_2; c_1, c_2} = E_{r_1, r_2; 0, c_2} \smallsetminus E_{r_1, r_2; 0, c_1}$.
    Hence
    \[
        \widehat{\mathbbm{1}}_{E_{r_1, r_2; c_1, c_2}} = \widehat{\mathbbm{1}}_{E_{r_1, r_2; 0, c_2}} - \widehat{\mathbbm{1}}_{E_{r_1, r_2; 0, c_1}},
    \]
    and it follows that for almost every $\vartheta \in M$,
    \[
        \lim_{T \to \infty} \frac{1}{T} \int_{0}^{T} \widehat{\mathbbm{1}}_{E_{r_1, r_2; c_1, c_2}}(g_t \Delta_\vartheta) \, \mathrm{d}t = |E_{r_1, r_2; c_1, c_2}|.
    \]

    Note that for all $\ell \in \mathbb{Z}_+$ one can find a finite collection of disjoint sets $\{E_{r^i_1, r^i_2; c^i_1, c^i_2}\}_{i \in I_\ell}$ so that
    \begin{gather}
        F_{r, c} \supseteq \bigsqcup_{i \in I_\ell} E_{r^i_1, r^i_2; c^i_1, c^i_2}, \label{eqFEleq}\\
        |F_{r, c}| - \frac{1}{\ell} \leq \sum_{i \in I_\ell} |E_{r^i_1, r^i_2; c^i_1, c^i_2}|. \label{eqFEvolleq}
    \end{gather}
    Let $M_\ell$ be a full measure subset of $M$ so that for all $\vartheta \in M_\ell$,
    \begin{equation}\label{eqgenEequidist}
        \lim_{T \to \infty} \frac{1}{T} \int_{0}^{T} \sum_{i \in I_\ell} \widehat{\mathbbm{1}}_{E_{r^i_1, r^i_2; c^i_1, c^i_2}}(g_t \Delta_\vartheta) \, \mathrm{d}t = \sum_{i \in I_\ell} |E_{r^i_1, r^i_2; c^i_1, c^i_2}|.
    \end{equation}
    Then $M_\infty := \bigcap M_\ell$ has full measure in $M$.
    Using~\eqref{eqFEleq} --~\eqref{eqgenEequidist} we see that for all $\vartheta \in M_\infty$, for all $\ell \in \mathbb{Z}_{+}$,
    \begin{align*}
        \liminf_{T \to \infty} \frac{1}{T} \int_{0}^{T} \widehat{\mathbbm{1}}_{F_{r, c}}(g_t \Delta_\vartheta) \, \mathrm{d}t &\geq \lim_{T \to \infty}\frac{1}{T} \int_{0}^{T} \sum_{i \in I_\ell} \widehat{\mathbbm{1}}_{E_{r^i_1, r^i_2; c^i_1, c^i_2}}(g_t \Delta_\vartheta) \, \mathrm{d}t \\
        &= \sum_{i \in I_\ell} |E_{r^i_1, r^i_2; c^i_1, c^i_2}| \\
        &\geq |F_{r, c}| - \frac{1}{\ell}.
    \end{align*}
    Therefore
    \begin{equation}\label{eqFequidistli}
        \liminf_{T \to \infty} \frac{1}{T} \int_{0}^{T} \widehat{\mathbbm{1}}_{F_{r, c}}(g_t \Delta_\vartheta) \, \mathrm{d}t \geq |F_{r, c}|,
    \end{equation}
    and it remains to note that~\eqref{eqFequidistls} and~\eqref{eqFequidistli} imply~\eqref{eqFequidist}.
\end{proof}

In the next lemma we show that for a Riemann integrable function $f$ on $K_S^{d}$, the function $\widehat{f}$ is in $C_\alpha(X)$.

\begin{lemma}\label{lemhatofriemannincal}
    For any $d$ and all sufficiently large $r$ there are constants $c_1, c_2$ such that if $\mathbbm{1}_{B_r}$ is the characteristic function of the open ball of radius $r$ around origin, then for all $\Lambda \in X$
    \begin{equation}\label{eqbralasymp}
        c_1 \alpha(\Lambda) \leq \widehat{\mathbbm{1}}_{B_r} \leq c_2 \alpha(\Lambda).
    \end{equation}
    In particular, for any Riemann integrable function $f : K_S^{d} \to \mathbb{R}, \widehat{f} \in C_\alpha(X)$.
\end{lemma}
\begin{proof}
    Every element of $X$ is a lattice in $\mathbb{R}^{dk}$ of fixed determinant.
    Applying~\cite{kleinbockshiweiss17} Lemma 5.1, we see that there exists $r_0$ such that~\eqref{eqbralasymp} holds for all $r > r_0$.

    Now given any Riemann integrable function $f : K_S^{d} \to \mathbb{R}$ there exist positive $r > r_0$ and $C$ such that $|f| \leq C \cdot \mathbbm{1}_{B_r}$.
    Hence~\ref{cal2} follows from~\eqref{eqbralasymp}.
    In order to prove~\ref{cal1}, let $S$ be the set of discontinuities of $f$ in $K_S^{d} \smallsetminus \{0\}$.
    Then $|S| = 0$ and it follows that the set $S'$ of discontinuities of $\widehat{f}$ is contained in $S'' := \{\Lambda : \Lambda \cap S \neq \varnothing\}$.
    For each $\vv{\bm{v}} \in \mathcal{O}^{d} \smallsetminus \{0\}$, the set of $g \in G$ such that $g\vv{\bm{v}} \in S$ has Haar measure zero in $G$, and hence $S''$, being a countable union of sets of measure zero, is measure zero.
    Therefore $\mu(S') = 0$.
\end{proof}

We are going to need Shi's equidistribution result (\cite{shi17}, Corollary 1.3) for number fields.
This follows from much general Theorem 1.2 of~\cite{shi17}.

\begin{theorem}\label{thmshitypeforks}
    Let $\Lambda \in X$.
    Then for almost every $\vartheta \in M$, $u(\vartheta)\Lambda$ is $(D^+, C_c(X))$-generic.
\end{theorem}
\begin{proof}
    In Theorem 1.2 of~\cite{shi17}, take $G = L = \mathrm{SL}_{d}(K_S)$, $\Gamma = \mathrm{SL}_{d}(\mathcal{O})$, $U = \mathcal{U}$ and $x = \Lambda$.
\end{proof}

Following~\cite{kleinbockshiweiss17}, we are now going to derive some general properties of convergence of measures on $X$.

\begin{lemma}\label{lemconvmeasi}
    Let $\psi \in C_\alpha(X)$ and let $\{\mu_i\}$ be a sequence of probability measures on $X$ such that $\mu_i \to \mu$ with respect to the weak-$\ast$ topology.
    Then for any non-negative $\varphi \in C_c(X)$ we have
    \[
        \lim_{i \to \infty} \int_{X} \varphi\psi \, \mathrm{d}\mu_i = \int_{X} \varphi\psi \, \mathrm{d}\mu.
    \]
\end{lemma}
\begin{proof}
    Decomposing $\psi$ into real and complex parts we can assume that $\psi$ is real valued.
    Using~\ref{cal1} and~\ref{cal2} we see that $\varphi\psi$ is bounded, compactly supported and continuous except on a set of measure zero.
    By using a partition of unity, without loss of generality one can assume that $\varphi$ is supported on a coordinate chart.
    Applying Lebesgue's criterion for Riemann integrability to $\varphi\psi$, we can write $\int_{X} \varphi\psi \, \mathrm{d}\mu$ as the limit of upper and lower Riemann sums.
    It follows that given $\varepsilon > 0$ there exist $h_1, h_2 \in C_c(X)$ such that $h_1 \leq \varphi\psi \leq h_2$ and
    \begin{equation}\label{eqlemconvmeasi}
        \int_{X} (h_2 - h_1) \, \mathrm{d}\mu \leq \varepsilon.
    \end{equation}
    Thus we have
    \begin{gather}
        \int_{X} h_1 \, \mathrm{d}\mu \leq \liminf_{i \to \infty} \int_{X} \varphi\psi \, \mathrm{d}\mu_i \leq \limsup_{i \to \infty} \int_{X} \varphi\psi \, \mathrm{d}\mu_i \leq \int_{X} h_2 \, \mathrm{d}\mu \label{eqlemconvmeasii} \\
        \int_{X} h_1 \, \mathrm{d}\mu \leq \int_{X} \varphi\psi \, \mathrm{d}\mu \leq \int_{X} h_2 \, \mathrm{d}\mu. \label{eqlemconvmeasiii}
    \end{gather}
    Since $\varepsilon$ was arbitrary, the lemma follows from~\eqref{eqlemconvmeasi} --~\eqref{eqlemconvmeasiii}.
\end{proof}

\begin{corollary}\label{corconvmeasi}
    Let the notation be as in Lemma~\ref{lemconvmeasi}.
    Assume that
    \begin{equation}\label{eqcorconvmeasi}
        \lim_{i \to \infty} \int_{X} \psi \, \mathrm{d}\mu_i = \int_{X} \psi \, \mathrm{d}\mu.
    \end{equation}
    Then for any $\varepsilon > 0$ there exists $i_0 > 0$ and $\varphi \in C_c(X)$ with $0 \leq \varphi \leq 1$ such that
    \begin{equation}\label{eqcorconvmeasii}
        {\left|\int_{X} (1 - \varphi) \psi \, \mathrm{d}\mu_i\right|} < \varepsilon
    \end{equation}
    for any $i \geq i_0$.
\end{corollary}
\begin{proof}
    Since $\psi \in L^1(X, \mu)$, there exists a compactly supported continuous function $\varphi : X \to [0, 1]$ such that
    \begin{equation}\label{eqcorconvmeasiii}
        {\left|\int_{X} (1 - \varphi) \psi \, \mathrm{d}\mu\right|} < \frac{\varepsilon}{3}.
    \end{equation}
    By Lemma~\ref{lemconvmeasi} and~\eqref{eqcorconvmeasi}, there exists $i_0 > 0$ such that for $i \geq i_0$
    \begin{gather}
        {\left|\int_{X} \varphi\psi \, \mathrm{d}\mu_i - \int_{X} \varphi\psi \, \mathrm{d}\mu\right|} < \frac{\varepsilon}{3} \label{eqcorconvmeasiv} \\
        {\left|\int_{X} \psi \, \mathrm{d}\mu_i - \int_{X} \psi \, \mathrm{d}\mu\right|} < \frac{\varepsilon}{3}. \label{eqcorconvmeasv}
    \end{gather}
    Therefore the corollary follows from~\eqref{eqcorconvmeasiii} --~\eqref{eqcorconvmeasv}.
\end{proof}

\begin{corollary}\label{corconvmeasii}
    Let the notation be as in Lemma~\ref{lemconvmeasi}.
    Assume that there exists a non-negative Riemann integrable function $f_0$ on $K_S^{d}$ such that $|\psi| \leq \widehat{f}_0$ and
    \begin{equation}\label{eqcorconvmeasvi}
        \lim_{i \to \infty} \int_{X} \widehat{f}_0 \, \mathrm{d}\mu_i = \int_{X} \widehat{f}_0 \, \mathrm{d}\mu.
    \end{equation}
    Then
    \[
        \lim_{i \to \infty} \int_{X} \psi \, \mathrm{d}\mu_i = \int_{X} \psi \, \mathrm{d}\mu.
    \]
\end{corollary}
\begin{proof}
    Using Lemma~\ref{lemconvmeasi}, Corollary~\ref{corconvmeasi} and~\eqref{eqcorconvmeasvi}, we have that for $\varepsilon > 0$ there exists $i_0 > 0$ and a continuous compactly supported function $\varphi : X \to [0, 1]$ such that for $i \geq i_0$
    \begin{gather*}
        {\left|\int_{X} \varphi\psi \, \mathrm{d}\mu_i - \int_{X} \varphi\psi \, \mathrm{d}\mu\right|} < \frac{\varepsilon}{3}, \\
        \int_{X} (1 - \varphi) \widehat{f}_0 \, \mathrm{d}\mu_i < \frac{\varepsilon}{3}, \\
        \int_{X}(1 - \varphi) \widehat{f}_0 \, \mathrm{d}\mu < \frac{\varepsilon}{3}.
    \end{gather*}
    Using $|\psi| \leq \widehat{f}_0$ and the above inequalities, we get that
    \[
        {\left|\int_{X} \psi \, \mathrm{d}\mu_i - \int_{X} \psi \, \mathrm{d}\mu\right|} < \varepsilon
    \]
    for $i > i_0$.
    Hence we are done.
\end{proof}

Now for given $\psi \in C_\alpha(X)$ we are going to apply Corollary~\ref{corofschi} and Corollary~\ref{corofschii} to find a function $f_0$ satisfying the hypothesis of Corollary~\ref{corconvmeasii}.
Theorem~\ref{thmshitypeforks} guarantees that there is a full measure subset $M'$ of $M$ so that $\Delta_\vartheta$ is $(D^+, C_c(X))$-generic for every $\vartheta \in M'$.

\begin{lemma}\label{lemequidistforannulus}
    For $r > dk$ let $A_r$ be the annular region defined by $A_r := \{\vv{\bm{z}} \in K_S^{d} : dk < \norm{\vv{\bm{z}}}_2 < r\}$.
    Then for every $\vartheta \in M'$
    \[
        \lim_{T \to \infty} \frac{1}{T} \int_{0}^{T} \widehat{\mathbbm{1}}_{A_r}(g_t \Delta_\vartheta) \, \mathrm{d}t = |A_r|.
    \]
\end{lemma}
\begin{proof}
    There exists $r', c' > 0$ such that $A_r \subseteq E_{r', c'} \cup F_{r', c'}$, hence it follows that $\widehat{\mathbbm{1}}_{A_r} \leq \widehat{\mathbbm{1}}_{E_{r', c'}} + \widehat{\mathbbm{1}}_{F_{r', c'}}$.
    Therefore Corollary~\ref{corofschi} and Corollary~\ref{corofschii} imply that~\eqref{eqcorconvmeasvi} holds for $f_0 = \mathbbm{1}_{E_{r', c'}} + \mathbbm{1}_{F_{r', c'}}$.
    Hence this lemma follows from Corollary~\ref{corconvmeasii}.
\end{proof}

\begin{lemma}\label{lemequidistforball}
    Let $B_r$ be the open ball of radius $r > 0$ around the origin in $K_S^{d}$.
    Then for every $\vartheta \in M'$
    \[
        \lim_{T \to \infty} \frac{1}{T} \int_{0}^{T} \widehat{\mathbbm{1}}_{B_r}(g_t \Delta_\vartheta) \, \mathrm{d}t = |B_r|.
    \]
\end{lemma}
\begin{proof}
    We first observe that there exists $s> 0$ any annular region of width $\geq s$ contains a lattice point of any unimodular lattice in $K_S^{d}$.
    Given $\Lambda \in X$ let
    \[
        \lambda_1(\Lambda) := \inf\{\lambda \geq 0 : \dim_{K}(\operatorname{span} (B_1 \cap \Lambda)) \geq 1\}.
    \]
    Using Theorem 1.2 of~\cite{kleinbockshitomanov17} we get that $\exists~s > 1$ such that $\lambda_1(\Lambda) \leq s-1$ for all $\Lambda \in X$.
    Therefore for any vector $\vv{\bm{v}}_{1} \in \Lambda$ with $\norm{\vv{\bm{v}}_{1}}_{2} = \lambda_1(\Lambda)$, any annular region of width $s$ would contain some integer multiple of $\vv{\bm{v}_{1}}$.
    So for any unimodular lattice $\Lambda$ in $K_S^{d}$ let $\vv{\bm{v}} \in \Lambda \cap (A_{dk+s+r} \smallsetminus A_{dk+r})$, and we have $\#(B_r \cap \Lambda) = \#((B_r + \vv{\bm{v}}) \cap \Lambda)$.
    It follows that for any lattice $\Lambda \in X$, the number of lattice points in $B_r$ is at most the number of lattice points in $A_{dk+s+2r}$, i.e., $\widehat{\mathbbm{1}}_{B_r} \leq \widehat{\mathbbm{1}}_{A_{dk+s+2r}}$.
    Therefore the conclusion of this lemma follows from Corollary~\ref{corconvmeasii} and Lemma~\ref{lemequidistforannulus}.
\end{proof}

\subsection*{Proof of Theorem~\ref{thmcalgenofDevth}}
It follows that for every $\vartheta \in M'$ and for all sufficiently large positive integers $r$
\[
    \lim_{T \to \infty} \frac{1}{T} \int_{0}^{T} \widehat{\mathbbm{1}}_{B_r}(g_t \Delta_\vartheta) \, \mathrm{d}t = \int_{X} \widehat{\mathbbm{1}}_{B_r} \, \mathrm{d}\mu.
\]
From~\ref{lemhatofriemannincal} it follows that for any $\psi \in C_\alpha(X)$ there exists $C > 0$ and $r \in \mathbb{Z}_+$ such that $|\psi(\Lambda)| \leq C \cdot \widehat{\mathbbm{1}}_{B_r}(\Lambda)$ for all $\Lambda \in X$.
It again follows from Corollary~\ref{corconvmeasii} that $\Delta_\vartheta$ is $(D^+, \psi)$-generic for every $\vartheta \in M'$, and hence the proof is complete. \qed\

\subsection*{Proof of Theorem~\ref{thmequidistonsphereforksk}}
Let $f_{A, B, r, c}$ be the characteristic function $E_{r, c}(A, B)$.
The assumption that the measures of boundaries of $A$ and $B$ are zero implies that $E_{r, c}(A, B)$ has measure zero boundary in $K_S^{d}$.
Hence $f_{A, B, r, c}$ is Riemann integrable on $K_S^{d}$.
Applying Theorem~\ref{thmcalgenofDevth} and~\eqref{eqsandwichequidistfhat3} we get
\[
    \lim_{T \to \infty} \frac{\#(E_{T, c}(A, B; \Delta_\vartheta))}{T} = \lim_{T \to \infty} \frac{1}{Tr} \int_{0}^{T} \widehat{f}_{A, B, r, c}(g_t \Delta_\vartheta) \, \mathrm{d}t = \frac{1}{r} \int_{X} \widehat{f}_{A, B, r, c} \, \mathrm{d}\mu = \frac{1}{r} |E_{r, c}(A, B)|.
\]
It now suffices to show that
\begin{equation}\label{eqetcabvol}
    r \cdot |E_{T, c}(A, B)| = T \cdot |E_{r, c}(A, B)|.
\end{equation}
For $T, r > 0$ such that $T/r = \ell \in \mathbb{Z}_+$, we have
\[
    E_{T, c}(A, B) = \bigsqcup_{j=0}^{\ell-1} g_{-rj}(E_{r, c}(A, B)),
\]
and hence~\eqref{eqetcabvol} follows for $T/r \in \mathbb{Z}_+$, since $\{g_t\}$ preserves $\lambda$.
From this one deduces~\eqref{eqetcabvol} for $T/r \in \mathbb{Q}$.
Finally for arbitrary $T, r > 0$ let $T_1, T_2 > 0$ be such that $T_1 < T < T_2$ and $T_1/r, T_2/r \in \mathbb{Q}$.
Then
\begin{gather*}
    E_{T_1, c}(A, B) \subseteq E_{T, c}(A, B) \subseteq E_{T_2, c}(A, B) \\
    \implies T_1 \cdot |E_{r, c}(A, B)| \leq r \cdot |E_{T, c}(A, B)| \leq T_2 \cdot |E_{r, c}(A, B)|.
\end{gather*}
Hence we get~\eqref{eqetcabvol} for all $T, r > 0$ by taking limits. \qed\


\subsection{Remark}
One could ask for the average value of $\#(E_{T, c}(A, B; \Delta_\vartheta))$ as $\Delta$ varies over $X$. It turns out that $\forall~\vartheta \in M$,
\begin{align*}
    \int_{X} \#(E_{T, c}(A, B; \Delta_\vartheta)) \, \mathrm{d}\mu(\Delta) &= \int_{X} \widehat{f}_{A, B, T, c} (u(\vartheta) \Delta) \, \mathrm{d}\mu(\Delta) \\
    &= \int_{K_S^{d}} f_{A, B, T, c} (u(\vartheta) \vv{\bm{z}}) \, \mathrm{d}\lambda(\vv{\bm{z}}) \\
    &= \int_{K_S^{d}} f_{A, B, T, c} (\vv{\bm{z}}) \, \mathrm{d}\lambda(\vv{\bm{z}}) \\
    &= |E_{T, c}(A, B)|.
\end{align*}



\subsection*{Acknowledgements} We thank the anonymous referee for helpful remarks.

\bibliographystyle{alpha}
\bibliography{AG_revised}

\begin{thebibliography}{AGGL19}

\bibitem[AG18]{athreyaghosh18}
J.~S. Athreya and A.~Ghosh.
\newblock {The Erd\H{o}s-Sz\"{u}sz-Tur\'{a}n distribution for equivariant
  processes}.
\newblock {\em {L'Enseignement Math\'{e}matique}}, {164}:{1--21}, {2018}.

\bibitem[AGGL19]{anghoshguanly16}
J.~An, A.~Ghosh, L.~Guan, and T.~Ly.
\newblock {Bounded orbits of Diagonalizable Flows on finite volume quotients of
  products of $\mathrm{SL}_{2}(\mathbb{R})$}.
\newblock {\em {Advances in Mathematics}}, {354}, {2019}.

\bibitem[AGP12]{athreyaghoshprasad12}
J.~S. Athreya, A.~Ghosh, and A.~Prasad.
\newblock {Ultrametric Logarithm Laws II}.
\newblock {\em {Monatsh. Math.}}, {167}:{333--356}, {2012}.

\bibitem[AGT14]{athreyaghoshtseng14}
J.~S. Athreya, A.~Ghosh, and J.~Tseng.
\newblock {Spherical averages of Siegel transforms for higher rank diagonal
  actions and applications}.
\newblock {preprint, \url{https://arxiv.org/abs/1407.3573}}, {2014}.

\bibitem[AGT15]{athreyaghoshtseng15}
J.~S. Athreya, A.~Ghosh, and J.~Tseng.
\newblock {Spiraling of approximations and spherical averages of Siegel
  transforms}.
\newblock {\em {J. Lond. Math. Soc. (2)}}, {91}({2}):{383--404}, {2015}.

\bibitem[APT16]{athreyaparrishtseng16}
J.~S. Athreya, A.~Parrish, and J.~Tseng.
\newblock {Ergodic theorem and Diophantine approximation for translation
  surfaces and linear forms}.
\newblock {\em {Nonlinearity}}, {29}({8}):{2173--2190}, {2016}.

\bibitem[EGL16]{einsiedlerghoshlytle16}
M.~Einsiedler, A.~Ghosh, and B.~Lytle.
\newblock {Badly approximable vectors, $C^1$ curves and number fields}.
\newblock {\em {Ergodic Theory Dynam. Systems}}, {36}({6}):{1851--1864},
  {2016}.

\bibitem[EMM98]{eskinmargulismozes98}
A.~Eskin, G.~A. Margulis, and S.~Mozes.
\newblock {Upper bounds and asymptotics in a quantitative version of the
  Oppenheim conjecture}.
\newblock {\em {Ann. of Math.}}, {147}:{93--141}, {1998}.

\bibitem[Gho19]{ghosh19}
A.~Ghosh.
\newblock {Topics in homogeneous dynamics and number theory}.
\newblock {preprint, \url{https://arxiv.org/abs/1901.02685}}, {2019}.

\bibitem[GR15]{ghoshroyals15}
A.~Ghosh and R.~Royals.
\newblock {An extension of the Khintchine-Groshev theorem}.
\newblock {\em {Acta Arith.}}, { 167}:{1--17}, {2015}.

\bibitem[KL16]{kleinbockly16}
D.~Kleinbock and T.~Ly.
\newblock {Badly approximable $S$-numbers and absolute Schmidt games}.
\newblock {\em {J. Number Theory}}, {164}:{13--42}, {2016}.

\bibitem[KL18]{kwonlim18}
S.~Kwon and S.~Lim.
\newblock {Equidistribution with an error rate and Diophantine approximation
  over a local field of positive characteristic}.
\newblock {\em {Discrete Contin. Dyn. Syst.}}, {38}:{169--186}, {2018}.

\bibitem[KST17]{kleinbockshitomanov17}
D.~Kleinbock, R.~Shi, and G.~Tomanov.
\newblock {$S$-adic version of Minkowski's geometry of numbers and Mahler's
  compactness criterion}.
\newblock {\em {J. Number Theory}}, {174}:{150--163}, {2017}.

\bibitem[KSW17]{kleinbockshiweiss17}
D.~Kleinbock, R.~Shi, and B.~Weiss.
\newblock {Pointwise equidistribution with an error rate and with respect to
  unbounded functions}.
\newblock {\em {Math. Ann.}}, {367}({1-2}):{857--879}, {2017}.

\bibitem[Ly16]{ly16}
T.~Ly.
\newblock {Diophantine approximation in algebraic number fields and flows on
  homogeneous spaces}.
\newblock {PhD thesis, Brandeis University}, {2016}.

\bibitem[Nak88]{nakada88}
Hitoshi Nakada.
\newblock {On metrical theory of Diophantine approximation over imaginary
  quadratic field}.
\newblock {\em {Acta Arith.}}, {51}:{399--403}, {1988}.

\bibitem[Sch60]{schmidt60}
W.~Schmidt.
\newblock {A metrical theorem in Diophantine approximation}.
\newblock {\em {Can. J. Math.}}, {12}:{619--31}, {1960}.

\bibitem[Shi17]{shi17}
R.~Shi.
\newblock {Pointwise equidistribution for one parameter diagonalizable group
  action on homogeneous space}.
\newblock {preprint, \url{https://arxiv.org/abs/1405.2067v3}}, {2017}.

\bibitem[Sul82]{sullivan82}
Dennis Sullivan.
\newblock {Disjoint spheres, approximation by imaginary quadratic numbers, and
  the logarithm law for geodesics}.
\newblock {\em {Acta Math.}}, { 149}:{215--237}, {1982}.

\bibitem[Wei46]{weil46}
A.~Weil.
\newblock {Sur quelques r\'esultats de Siegel}.
\newblock {\em {Summa Brasil. Math.}}, {1}:{21--39}, {1946}.

\end{thebibliography}

\end{document}